\documentclass[a4paper,12pt]{article}
\usepackage[utf8]{inputenc}
\usepackage{amsthm}
\usepackage{amsfonts}
\usepackage{amssymb}
\usepackage{hyperref}
\usepackage{graphicx}
\usepackage{rotating}
\usepackage{wrapfig}
\usepackage[all]{xy}
\usepackage{pgf,tikz}
\usetikzlibrary{arrows}
\usetikzlibrary[patterns]

\newcommand{\N}{\mathbb{N}}

\newcommand{\R}{\mathbb{R}}

\newcommand{\U}{\mathfrak{U}}
\newcommand{\B}{\mathfrak{B}}

\newtheorem{prop}{Proposition}[section]
\newtheorem{cor}[prop]{Corollary}

\newtheorem{teo}[prop]{Theorem}

\newtheorem{lema}[prop]{Lemma}

\theoremstyle{definition}

\newtheorem*{ex}{Example}
\newtheorem{defi}[prop]{Definition}
\newtheorem*{obs}{Remark}
\newtheorem*{acknowledgements}{Acknowledgements}

\title{Some properties that are preserved by transferring boundary functors}
\author{Lucas H. R. de Souza}

\begin{document}

\DeclareGraphicsExtensions{.pdf,.jpg,.mps,.png,}

\maketitle

\def\eod{\hfill$\square$}

\begin{abstract}If a Hausdorff locally compact paracompact space has a coarse structure, then there is a family of well behaved compactifications associated to it. If there are two of these spaces, $X$ and $Y$, with a good coarse equivalence, then there is a correspondence between these families of compactifications of $X$ and $Y$. On the other hand, if a group $G$ has a properly discontinuous cocompact action on a Hausdorff locally space $X$, then there is also a correspondence between nice compactifications of $G$ and nice compactifications of $X$. In this paper we show that when there are both concepts involved (coarse structure and group action), then both correspondences of families of compactifications agree.

We also prove that these correspondences must preserve some geometric properties of the compactifications.
\end{abstract}

\let\thefootnote\relax
\let\thefootnote\relax\footnote{Keywords: Perspective compactification, coarse geometry, functors that transfer boundaries, $\delta$-hyperbolic space, relatively hyperbolic group, Floyd compactification.}

\tableofcontents

\section*{Introduction}

There are two good notions of compactifications of locally compact Hausdorff spaces that appear in different contexts.  One of them appears when the space has a coarse structure compatible with its topology (we can think of a proper metric space up to quasi-isometry as an example). We call them coarse geometric perspective compactifications. The other one appears when a group acts properly discontinuously and cocompactly on the locally compact space. We call then group theoretic perspective compactifications. We call both perspective compactifications, when there is no ambiguity. The intuition behind both is that bounded things become small when they go to infinity. 

If $X$ and $Y$ are quasi-isometric AR proper metric spaces that a group $G$ acts geometrically, then Guilbault and Moran constructed a correspondence of $E\mathcal{Z}$-structures of $G$ with respect to $X$ and $Y$ ($E\mathcal{Z}$-version of the Boundary Swapping Theorem, consequence of Corollary 7.3 of \cite{GM}). There is a natural generalization of their correspondence for two coarse equivalent locally compact paracompact spaces with coarsely connected proper coarse structures (Theorems 7.32 and 7.37 of \cite{So}). On the other hand, if a group $G$ acts properly discontinuously and cocompactly on two Hausdorff locally compact spaces $X$ and $Y$, then there is a correspondence between group theoretic perspective compactifications of $X$ and $Y$ (Theorem 3.2 of \cite{So2}). This generalizes a construction given by Gerasimov. His Attractor-Sum Theorem (Theorem 8.3.1 of \cite{Ge2}) shows that if a group acts properly discontinuously and cocompactly on a space $X$ and it acts with the convergence property on a space $Z$, then there is a unique compactification of $G$ with boundary $Z$ such that the induced action on this new space still has the convergence property. He also showed on the same theorem that there is a unique compactification of $X$ with boundary $Z$ such that the induced action on this new space still have the convergence property. This construction from the compactification of $G$ to the compactification of $X$ is that one that we generalize to perspective compactifications.

If a group $G$ has a properly discontinuous cocompact action on a Hausdorff locally compact paracompact space $X$, then there are natural coarse structures for $G$ and $X$ (\textbf{Definition \ref{coarseactionstructure}}) that are coarse equivalent (\textbf{Proposition \ref{milnorsvarccoarse}}). This is a generalization of the Milnor-$\check{S}$varc Lemma. The main goal of this paper is to show that the notions of the geometric coarse perspective compactifications and group theoretic perspective compactifications coincide in this case of $G$ and $X$, considering these natural coarse structures. Moreover, if $G$ is countable, $X$ has a countable basis and we consider only metrizable equivariant compactifications, then both correspondences of compactifications (given by Theorems 3.37 of \cite{So} and Corollary 3.5 of \cite{So2}) also coincide.

Formally, we have the following:

Let $G$ be a countable group, $X$ a Hausdorff locally compact space and $\varphi: G \curvearrowright X$ a properly discontinuous cocompact action. We define $EMPers_{0}(\varphi)$ the category whose objects are group theoretic perspective compactifications of $X$ with respect to $\varphi$ that are metrizable, the action $\varphi$ extends continuously to the compactification and the morphisms are equivariant continuous maps that are the identity on $X$. If $X = G$ and $\varphi$ is the left multiplication action, then we denote $EMPers_{0}(\varphi)$ by $EMPers_{0}(G)$. Corollary 3.5 of \cite{So2} gives an isomorphism of categories $\Pi: EMPers_{0}(G) \rightarrow EMPers_{0}(\varphi)$.

Let $(X, \varepsilon)$ be a locally compact paracompact Hausdorff space with a  coarsely connected proper coarse structure. We define $MPers(\varepsilon)$ the category whose objects are coarse geometric perspective compactifications of $X$ that are metrizable and morphisms are continuous maps that are the identity on $X$. Theorem 7.37 of \cite{So} says that a coarse equivalence $\pi: (Y,\eta) \rightarrow (X,\varepsilon)$ induces an isomorphism of categories $\Pi': MPers(\varepsilon) \rightarrow MPers(\eta)$.

If a countable group $G$ acts properly discontinuously and cocompactly on a Hausdorff locally compact paracompact space $X$, we can consider natural coarse structures $\varepsilon_{G}$ for $G$ and $\varepsilon_{\varphi}$ for $X$ (\textbf{Definition \ref{coarseactionstructure}}). Then an equivariant compactification of $X$ is group theoretic perspective if and only if it is coarse geometric perspective (\textbf{Proposition \ref{coarsegroup}}). Our version of the Milnor-$\check{S}$varc Lemma says that if we choose a point $x_{0} \in X$, then the map $\varphi_{x_{0}}: G \rightarrow X$ that sends an element $g$ to $gx_{0}$ is a coarse equivalence between these coarse structures (\textbf{Proposition \ref{milnorsvarccoarse}}). Then we have the following:

\

\

$\!\!\!\!\!\!\!\!\!\!\!\!$ \textbf{Theorem \ref{main}} Let $\Pi: EMPers_{0}(G) \rightarrow EMPers_{0}(\varphi)$ be the functor given by the action $\varphi$ and $\Pi': MPers(\varepsilon_{G}) \rightarrow MPers(\varepsilon_{\varphi})$ be the functor given by the quasi-isometry $\varphi_{x_{0}}$. Then $\Pi'|_{EMPers_{0}(G)} = \Pi$.

\


We also reinterpret some well known results by proving that the functors showed above preserve some classes of compactifications:

\begin{enumerate}
    \item If two proper $\delta$-hyperbolic geodesic spaces are quasi-isometric, then the induced functor sends the hyperbolic compactification to the hyperbolic compactification (\textbf{Theorem \ref{preserveshyperboliccompactification}}).
    \item If $G$ and $H$ are finitely generated groups that are quasi-isometric, then the induced functor sends relatively hyperbolic compactifications of $G$ with respect to subgroups that are not relatively hyperbolic to relatively hyperbolic compactifications of $H$ with respect to subgroups that are not relatively hyperbolic (\textbf{Proposition \ref{preserviingrelativelyhyperboliccompactification}}).
    \item If two locally finite graphs $\Gamma_{1}$ and $\Gamma_{2}$ are $\alpha$-quasi-isometric, for some function $\alpha$ which has nice relations with two Floyd maps $f_{1}$ and $f_{2}$, then the induced functor sends the Floyd compactification of $\Gamma_{1}$ with respect to $f_{1}$ to the Floyd compactification of $\Gamma_{2}$ with respect to $f_{2}$ (\textbf{Corollary \ref{preservingfloydcompactifications}}).
\end{enumerate}

\begin{acknowledgements}This paper contains part of my PhD thesis. It was written under the advisorship of Victor Gerasimov, to whom I am grateful. I also would like to thank Christopher Hruska, who showed me Groff's paper \cite{Gro} with that well behavior between relatively hyperbolic pairs and quasi-isometries.
\end{acknowledgements}

\section{Preliminaries}

\subsection{Artin-Wraith glueings}

This theory is covered in \cite{So}. We need to use it in the following sections just as a language to define some compactifications of spaces.

\begin{defi}Let $X, Y$ be topological spaces and $f: Closed(X) \rightarrow Closed (Y)$ a map such that $\forall A,B \in Closed(X), \ f(A\cup B) = f(A)\cup f(B)$ and $f(\emptyset) = \emptyset$ (we will say that such map is admissible). We will give a topology for $X\dot{\cup} Y$. Let's declare as a closed set $A \subseteq X\dot{\cup} Y$ if $A \cap X \in Closed(X), \ A \cap Y \in Closed(Y)$ and $f(A\cap X) \subseteq A$. Therefore, let's denote by $\tau_{f}$ the set of the complements of this closed sets and $X+_{f}Y = (X\dot{\cup} Y,\tau_{f})$.
\end{defi}

Every space $X$ that is a union of disjoint subspaces $A$ and $B$, where $A$ is open on $X$, can be recovered uniquely as an Artin-Wraith glueing of $A$ and $B$.

\begin{defi}Let $X+_{f} \! Y$ and $Z+_{h} \! W$ be topological spaces and continuous maps $\psi: X \rightarrow Z$ and $\phi: Y \rightarrow W$. We define $\psi + \phi: X+_{f}Y \rightarrow Z+_{h}W$ by $(\psi + \phi)(x) = \psi(x)$ if $x \in X$ and $\phi(x)$ if $x \in Y$. If $G$ is a group, $\psi: G \curvearrowright X$ and $\phi: G \curvearrowright Y$, then we define $\psi+\phi: G \curvearrowright X+_{f}Y$ by $(\psi+\phi)(g,x) = \psi(g,x)$ if $x \in X$ and $\phi(g,x)$ if $x \in Y$.
\end{defi}

\begin{prop}\label{continuousidentity} (Corollary 2.15 of \cite{So}) Let $X+_{f}Y$, $X+_{g}Y$ be topological spaces. Then, the map $id: X+_{f}Y \rightarrow X+_{g}Y$ is continuous if and only if $\forall A \in Closed (X)$, $f(A) \subseteq g(A)$.
\end{prop}

\begin{prop}\label{compactglue} (Proposition 2.12 of \cite{So}) Let $X,Y$ be topological spaces with $Y$ compact and $f$ an admissible map. Then $X+_{f}Y$ is compact if and only if $\forall A \in Closed(X)$ non compact, $f(A) \neq \emptyset$.
\end{prop}

\begin{defi}Let $X+_{f}W$, $Y$ and $Z$ be topological spaces, $\pi: Y \rightarrow X$ and $\varpi: Z \rightarrow W$ be two maps. We define the pullback of $f$ with respect to $\pi$ and $\varpi$ by $f^{\ast}(A) = Cl_{Z}(\varpi^{-1}(f(Cl_{X}\pi(A))))$.
\end{defi}

\begin{prop}\label{univproppullback} (Propositions 2.23 and 2.24 of \cite{So}) Let $X+_{f}W$, $Y$ and $Z$ be topological spaces, $\pi: Y \rightarrow X$ and $\varpi: Z \rightarrow W$ be two continuous maps. Then $\pi+\varpi: Y+_{f^{\ast}}Z \rightarrow X+_{f}W$ is continuous. Moreover, if we have $Y+_{f'}Z$, for some admissible map $f'$, such that $\pi+\varpi: Y+_{f'}Z \rightarrow X+_{f}W$ is continuous, then, $id_{Y}+id_{Z}: Y+_{f'}Z \rightarrow Y+_{f^{\ast}}Z$ is continuous.
\end{prop}

\begin{prop}\label{compositionpullback} (Proposition 2.31 of \cite{So}) Let $X+_{f}W$, $Y$, $Z$, $U$ and $V$ be topological spaces, $\pi: Y \rightarrow X$, $\varpi: Z \rightarrow W$, $\rho: U \rightarrow Y$ and $\varrho: V \rightarrow Z$ be four maps. Then $f^{\ast\ast} \subseteq (f^{\ast})^{\ast}$, where $f^{\ast\ast}$ is the pullback of $f$ by the maps $\pi \circ \rho$ and $\varpi \circ \varrho$ and  $(f^{\ast})^{\ast}$ is the pullback of $f^{\ast}$ by the maps $\rho$ and $\varrho$.
\end{prop}

\begin{prop}\label{eightlemma}Let $X+_{f}W$, $Y$ and $Z$ be topological spaces, $\pi: Y \rightarrow X$ and $\varpi: Z \rightarrow W$ be two continuous maps. Let $f^{\ast}$ be the pullback of $f$ with respect to $\pi$ and $id_{W}$, $(f^{\ast})^{\ast}$ be the pullback of $f^{\ast}$ with respect to $id_{Y}$ and $\varpi$ and $f^{\ast\ast}$ the pullback of $f$ with respect to $\pi$ and $\varpi$. Similarly, let $f'$ be the pullback of $f$ with respect to $id_{X}$ and $\varpi$, $(f')'$ be the pullback of $f'$ with respect to $\pi$ and $id_{Z}$ and $f'' = f^{\ast\ast}$. Then all maps in the following diagram are continuous: 

$$ \xymatrix{  & Y+_{f^{\ast\ast}}Z \ar[r]_{id+id} & Y+_{(f^{\ast})^{\ast}}Z \ar[r]_{id+\varpi} & Y+_{f^{\ast}}W \ar[rd]^{\pi+id} &  \\
                Y+_{g}Z \ar[rrrr]^{\pi+\varpi} \ar[ru]^{id+id} \ar[rd]_{id+id}& & & & X+_{f}W   \\
              & Y+_{f''}Z \ar[r]_{id+id} & Y+_{(f')'}Z \ar[r]_{\pi+id} & X+_{f'}Z \ar[ru]_{id+\varpi} & } $$
\end{prop}

\begin{obs}In particular, we use later that the map $id+\varpi: Y+_{g}Z \rightarrow Y+_{f^{\ast}}W$ is continuous.
\end{obs}

\begin{proof}By \textbf{Proposition \ref{univproppullback}}, the maps $\pi+id: Y+_{f^{\ast}}W \rightarrow X+_{f}W$, $id+\varpi: Y+_{(f^{\ast})^{\ast}}Z \rightarrow  Y+_{f^{\ast}}W$ and $id+id: Y+_{g}Z \rightarrow Y+_{f^{\ast\ast}}Z$ are continuous. By \textbf{Propositions \ref{compositionpullback}} and \textbf{\ref{continuousidentity}}, the map $id+id: Y+_{f^{\ast\ast}}Z \rightarrow Y+_{(f^{\ast)^{\ast}}}Z$ is continuous.

The other side of the diagram is analogous.
\end{proof}

\subsection{Coarse geometry}

The next definitions and propositions follow John Roe's book \cite{Ro}.

\begin{defi}Let $X$ be a set. A coarse structure on $X$ is a set $\varepsilon \subseteq X \times X$ satisfying:

\begin{enumerate}
    \item The diagonal $\Delta X \in \varepsilon$,
    \item If $e \in \varepsilon$ and $e' \subseteq e$, then $e' \in \varepsilon$,
    \item If $e,e' \in \varepsilon$ then $e\cup e' \in \varepsilon$,
    \item If $e \in \varepsilon$  then $e^{-1} = \{(a,b): (b,a) \in e\} \in \varepsilon$,
    \item If $e,e' \in \varepsilon$  then $e' \circ e = \{(a,b): \exists c \in X: (a,c) \in e, (c,b) \in e'\} \in \varepsilon$.
\end{enumerate}

\end{defi}

\begin{prop}(Proposition 2.12 of \cite{Ro}) Let $X$ be a set and $\{\varepsilon_{\gamma}\}_{\gamma \in \Gamma}$ a set of coarse structures for $X$. Then $\varepsilon = \bigcap_{\gamma \in \Gamma}\varepsilon_{\gamma}$ is a coarse structure for $X$.
\end{prop}

\begin{defi}Let $X$ be a set and $\mathcal{A} \subseteq \mathcal{P}(X \times X)$. The coarse structure generated by $\mathcal{A}$ is the intersection of all coarse structures that contains $\mathcal{A}$.
\end{defi}

\begin{obs}The maximal coarse structure contains $\mathcal{A}$, so this intersection is not empty.
\end{obs}

\begin{prop}Let $X$ be a set, $\mathcal{A} \subseteq \mathcal{P}(X \times X)$ and $\varepsilon$ the coarse structure generated by $\mathcal{A}$. Suppose that $\mathcal{A}$ satisfies the following conditions:

\begin{enumerate}
    \item $\exists e \in \mathcal{A}$ such that $\Delta X \subseteq e$,
    \item If $e,e' \in \mathcal{A}$ then there exists $e'' \in \mathcal{A}$ such that $e\cup e' \subseteq e''$,
    \item If $e \in \mathcal{A}$  then $e^{-1} \in \mathcal{A}$,
    \item If $e,e' \in \mathcal{A}$  then $e' \circ e \in \mathcal{A}$.
\end{enumerate}

Then for every $e \in \varepsilon$, there exists $e' \in \mathcal{A}$ such that $e \subseteq e'$.
\end{prop}

\begin{proof}Let $\varepsilon' = \{e \subseteq X\times X: \exists e' \in \mathcal{A}: e \subseteq e'\}$. By the definition, this is a subset of $\varepsilon$. By the first condition, $\Delta X \in \varepsilon'$. 

Let $e \in \varepsilon$ and $e' \subseteq e$. There exists $e'' \in \mathcal{A}$ such that $e' \subseteq e \subseteq e''$, which implies that $e' \in \varepsilon'$.

Let $e,e' \in \varepsilon'$. There exists $e'',e''' \in \mathcal{A}$ such that $e \subseteq e''$ and $e' \subseteq e'''$. Then $e\cup e' \subseteq e'' \cup e'''  \subseteq e^{4}$, for some $e^{4} \in \mathcal{A}$, which implies that $e \cup e' \in \varepsilon'$. We have also that $e\circ e' \subseteq e'' \circ e''' \in  \mathcal{A}$, which implies $e\circ e' \in \varepsilon$, 

Let $e\in \varepsilon'$. There exists $e' \in \mathcal{A}$ such that $e \subseteq e'$. Then $e^{-1} \subseteq e'^{-1} \in \mathcal{A}$, which implies that $e^{-1} \in \varepsilon'$.

So $\varepsilon'$ is a coarse structure that contains $\mathcal{A}$, which implies that $\varepsilon \subseteq \varepsilon'$ and then $\varepsilon = \varepsilon'$.
\end{proof}

In this case we say that $\mathcal{A}$ is a basis for the coarse structure $\varepsilon$.

\begin{prop}(Proposition 2.16 of \cite{Ro}) Let $B$ be a subset of a coarse space $(X,\varepsilon)$. Then $B\times B \in \varepsilon$ if and only if there exists $b \in X$ such that $B\times \{b\} \in \varepsilon$.
\end{prop}

When we have those equivalent conditions on $B$, we say that $B$ is a bounded subset of $X$.

\begin{defi}Let $X$ be a topological space. A subset of $X$ is topologically bounded (or relatively compact) if its closure on $X$ is compact. We say that $(X,\varepsilon)$ is a proper coarse space if the coarse structure has a neighbourhood of $\Delta X$ and every bounded subset of $X$ is topologically bounded.
\end{defi}

\begin{defi}A coarse space $(X,\varepsilon)$ is coarsely connected if for every $(x,y) \in X \times X$, $\exists e \in \varepsilon:$ $(x,y) \in e$.
\end{defi}

\begin{defi}Let $X$ be a set, $Y \subseteq X$ and $u \subseteq X \times X$. We define the $u$-neighbourhood of $Y$ by $\mathfrak{B}(Y,u) = \{x \in X: \exists y \in Y: \ (x,y)\in u\}$. Let $X$ be a topological space. A subset $e \subseteq X\times X$ is proper if $\forall B \subseteq X$ topologically bounded, $\B(B,e)$ and $\B(B, e^{-1})$ are topologically bounded.
\end{defi}

\begin{defi}Let $(M,d)$ be a metric space. The bounded coarse structure associated to the metric $d$ is the collection of sets $e \subseteq X\times X$ such that $\sup\{d(x,y): (x,y)\in e\} < \infty$. We denote this coarse structure by $\varepsilon_{d}$.
\end{defi}

\begin{prop}\label{propercontrolled}(Proposition 2.23 of \cite{Ro}) Let $(X,\varepsilon)$ be a coarsely connected proper coarse space. A subset of $X$ is bounded if and only if it is topologically bounded. Moreover, every element of $\varepsilon$ is proper.
\end{prop}

\begin{defi}Let $f: (X,\varepsilon) \rightarrow (Y,\zeta)$ be a map. We say that $f$ is bornologous if $\forall e \in \varepsilon$, $f(e)\in \zeta$. We say that $f$ is proper if $\forall B \subseteq Y$ bounded, $f^{-1}(B)$ is bounded. We say that $f$ is coarse if it is proper and bornologous.
\end{defi}

\begin{prop}\label{coarsecomposition}Let $f: (X,\varepsilon) \rightarrow (Y,\zeta)$. If $e,e' \in \varepsilon$ and $f(e),f(e') \in \zeta$, then $f(e' \circ e) \in \zeta$.
\end{prop}

\begin{proof}Let $(a,b) \in e' \circ e$. Then $\exists c \in X:$ $(a,c) \in e$, $(c,b) \in e'$, which implies that $(f(a),f(c)) \in f(e)$, $(f(c),f(b)) \in f(e')$, which implies that $(f(a),f(b)) \in f(e') \circ f(e)$. Then $f(e' \circ e)  \subseteq f(e') \circ f(e) \in \zeta$, which implies that $f(e' \circ e) \in \zeta$.
\end{proof}

\begin{prop}\label{coarsebasis} (Proposition 1.19 of \cite{Gra}) Let $f: (X,\varepsilon) \rightarrow (Y,\zeta)$ and $\mathcal{A}$ a basis for $\varepsilon$. If $\forall e \in \mathcal{A}$, $f(e) \in \zeta$, then $f$ is bornologous.
\end{prop}

\begin{defi}Let $S$ be a set and $(X,\varepsilon)$ a coarse space. Two maps $f,g: S \rightarrow X$ are close if $\{(f(s),g(s)): s \in S\} \in \varepsilon$.
\end{defi}

\begin{defi}Let $(X,\varepsilon)$ be a coarse space and $A \subseteq X$. The subspace coarse structure is defined by $\varepsilon|_{A} = \{e \in \varepsilon: e \subseteq A\times A\}$.
\end{defi}

\begin{defi}Two coarse spaces $(X,\varepsilon)$ and $(Y,\zeta)$ are coarsely equivalent if there exists two coarse maps $f: X \rightarrow Y$ and $g: Y \rightarrow X$ that are quasi-inverses, i.e. $f\circ g$ is close to $id_{Y}$ and $g \circ f$ is close to $id_{X}$. A coarse map $f: (X,\varepsilon) \rightarrow (Y,\zeta)$ is a coarse embedding if it is a coarse equivalence between $X$ and $f(X)$, with the subspace coarse structure.
\end{defi}

\begin{prop}Let $(X,\varepsilon)$ be a coarse space and $A \subseteq X$. The inclusion map $\iota: A \rightarrow X$ is a coarse embedding. \eod
\end{prop}

\begin{defi}Let $(X,\varepsilon)$ be a coarse space, $A \subseteq X$ and $e \in \varepsilon$. We say that $A$ is $e$-quasi-dense if $\B(A,e) = X$. We say that $A$ is quasi-dense if it is $e$-quasi-dense for some $e \in \varepsilon$.
\end{defi}

\begin{prop}Let $(X,\varepsilon)$ be a coarsely connected coarse space and $A \subseteq X$. If $A$ is quasi-dense, then the inclusion map $\iota: A \rightarrow X$ is a coarse equivalence.
\end{prop}

\begin{proof}Let $e \in \varepsilon$ such that $A$ is $e$-quasi-dense. Let $f: X \rightarrow A$ that chooses $f(x)$ such that $(f(x),x) \in e$ (it is possible since $A$ is $e$-quasi-dense). We can also suppose that $f \circ \iota = id_{A}$.

Let $B \subseteq A$ be a bounded set. Then $f^{-1}(B) \subseteq \B(B,e^{-1})$, which is bounded since $X$ is coarsely connected. Then $f$ is proper.

Let $e' \in \varepsilon$. If $(a,b) \in e'$, then $(f(a),f(b)) \in e^{-1} \circ e' \circ e \in \varepsilon$, which implies that $f(e') \in \varepsilon$. Then $f$ is bornologous.

Let $x \in X$. Then $(\iota \circ f(x), x) = (f(x),x) \in e$, which implies that $\iota \circ f$ and $id_{X}$ are close. Since $f \circ \iota = id_{A}$, it follows that $\iota$ and $f$ are coarse inverses. Thus $\iota$ is a coarse equivalence.
\end{proof}

\begin{cor}\label{quasidense}Let $f: (X,\varepsilon) \rightarrow (Y,\zeta)$ be a coarse embedding, with $Y$ coarsely connected. If $f(X)$ is quasi-dense, then $f$ is a coarse equivalence. \eod
\end{cor}

\begin{prop}(Theorem 2.27 of \cite{Ro}) Let $X$ be a locally compact paracompact space, $W$ a compactification of $X$ (i.e. $W$ is Hausdorff compact, contains $X$ as a subspace and $X$ is dense in $W$), $\partial X = W - X$ and $e \subseteq X \times X$. The following conditions are equivalent:

\begin{enumerate}
    \item $Cl_{W^{2}}(e) \cap (W^{2} - X^{2}) \subseteq \Delta \partial X$.
    \item $e$ is proper and if  $\{(x_{\gamma},y_{\gamma})\}_{\gamma \in \Gamma}$ is a net contained in $e$ such that $\lim x_{\gamma} = x \in \partial X$, then $\lim y_{\gamma} = x$.
    \item $e$ is proper and $\forall x \in \partial X$, $\forall V$ neighbourhood of $x$ in $W$, there exists $U$ a neighbourhood of $x$ such that $U \subseteq V$ and $e \cap (U \times (X-V)) = \emptyset$.
\end{enumerate}

We say that $e$ is perspective if it satisfies these equivalent definitions. We denote by $\varepsilon_{W}$ the set of perspective subsets of $X \times X$. Then $(X,\varepsilon_{W})$ is a coarsely connected proper coarse space.
\end{prop}

\subsection{Coarse geometric perspectivity}

\begin{defi}\label{comp}Let $X$ be a locally compact Hausdorff space. Let $Comp(X)$ be the category whose objects are compact spaces of the form $X+_{f}W$, where $W$ is a Hausdorff compact space and morphisms are continuous maps that are the identity on $X$.
\end{defi}

\begin{prop}\label{pullbackdecompactificacoes} (Proposition 3.13 of \cite{So}) Let $X, Y$ be locally compact Hausdorff spaces and $\pi: X \rightarrow Y$ a proper map. Then, the map $\Pi: Comp(Y) \rightarrow Comp(X)$ such that $\Pi(Y+_{f}Z) = X+_{f^{\ast}}Z$ and, for $id+\varpi: Y+_{f}Z \rightarrow Y+_{g}W$, $\Pi(id+\varpi) = id+\varpi: X+_{f^{\ast}}Z \rightarrow X+_{g^{\ast}}W$ is a functor.
\end{prop}

\begin{obs}It is possible that this functor sends Hausdorff spaces to non Hausdorff spaces. It all depends on the behavior of the map $\pi$. However, the next proposition says that this is not a problem for coarse geometry.
\end{obs}

Let $X$ be a locally compact paracompact Hausdorff space and let $W = X+_{f}Y$ be a Hausdorff compactification of $X$. We denote by $\varepsilon_{f}$ the coarse structure on $X$ induced by $W$ (instead of $\varepsilon_{W}$).

\begin{defi}\label{coarseperspectivity} Let $X$ be a locally compact paracompact Hausdorff space, $\varepsilon$ be a coarsely connected proper coarse structure of $X$ and $X+_{f}W$ a Hausdorff compactification of $X$ . We say that $X+_{f}W$ is perspective if $\varepsilon_{f} \supseteq \varepsilon$. Let $Pers(\varepsilon)$ be the full subcategory of $Comp(X)$ whose objects are perspective compactifications of $(X,\varepsilon)$. Let $MPers(\varepsilon)$ be the full subcategory of $Pers(\varepsilon)$ such that the objects are metrizable spaces.
\end{defi}

\begin{obs}Such compactifications are called coarse in \cite{Ro}.
\end{obs}

\begin{prop}\label{Teorema1} (Theorems 3.32 and 3.37 of \cite{So}) Let $(X,\varepsilon)$ and $(Y, \zeta)$ be locally compact paracompact Hausdorff space with coarsely connected proper coarse structures and $\pi: Y \rightarrow X$ a coarse equivalence with quasi-inverse $\varpi$. Consider the pullback functors $\Pi$ and $\Lambda$, with respect to $\pi$ and $\varpi$, respectively. Then we have the following:

\begin{enumerate}
    \item If $\pi$ and $\varpi$ are continuous, then the restrictions of $\Pi$ and $\Lambda$  are isomorphisms between the categories $Pers(\varepsilon)$ and $Pers(\zeta)$ that are inverses to each other.
    \item If $X$ and $Y$ have countable basis, then the restrictions of $\Pi$ and $\Lambda$ are isomorphisms between $MPers(\varepsilon)$ and $MPers(\zeta)$ that are inverses to each other.
\end{enumerate}

\end{prop}

\begin{obs}Our construction is the same as the one used on Theorem 7.1 of \cite{GM} for the case where the spaces $X$ and $Y$ are uniformly contractible ANR metric spaces. On the language that we are using,  Theorem 7.1 of \cite{GM} says that the functors $\Pi$ and $\Lambda$ sends controlled $\mathcal{Z}$-compactifications to controlled $\mathcal{Z}$-compactifications (controlled in their sense is equivalent to be, in our sense, perspective with respect to the bounded coarse structure).
\end{obs}

\begin{prop}\label{quotientpersp} (Proposition 3.48 of \cite{So}) Let $(X,\varepsilon)$ be a locally compact paracompact Hausdorff space with a proper coarsely connected coarse structure and $X+_{f}Y \in Pers(\varepsilon)$. If $X+_{g}Z$ is a compactification of $X$ and there exists a continuous map $id+\pi: X+_{f}Y \rightarrow X+_{g}Z$, then $X+_{g}Z \in Pers(\varepsilon)$.
\end{prop}

\begin{prop}\label{closemapsequalpullbacks}  (Proposition 3.28 of \cite{So}) Let $(X,\varepsilon)$, $(Y,\eta)$ be locally compact paracompact Hausdorff spaces with proper coarsely connected coarse structures and $X+_{f}Z\in Pers(\varepsilon)$. If the coarse maps $\pi,\pi': Y \rightarrow X$ are close, then the pullback with respect to $\pi$ and $id_{Z}$ and the pullback with respect to $\pi'$ and $id_{Z}$ are equal. Moreover, the induced functors with respect to $\pi$ and $\pi'$ are the same.
\end{prop}

\begin{prop}\label{closesameboundary}  (Proposition 3.27 of \cite{So}) Let $X+_{f}W \in Pers(\varepsilon)$ and $A,B \subseteq X$. If there exists an element $e \in  \varepsilon$ such that $A \subseteq \B(B,\varepsilon)$, then $Cl_{X+_{f}W}(A) - X \subseteq Cl_{X+_{f}W}(B) - X$. \eod
\end{prop}

\begin{obs}In particular, there is this well known phenomenon that if $X$ is a proper metric space and $A$ and $B$ have finite Hausdorff distance, then $Cl_{X+_{f}W}(A) - X = Cl_{X+_{f}W}(B) - X$.
\end{obs}

\subsection{Group theoretic perspectivity}

\begin{defi}Let $X$ be a set, $Y \subseteq X$ and $u \subseteq X \times X$. We say that $Y$ is $u$-small if $Y\times Y \subseteq u$. We define the set of $u$-small subsets by $Small(u)$.
\end{defi}

\begin{defi}Let $G$ be a group, $X$ a Hausdorff locally compact space and $\varphi: G \curvearrowright X$ a properly discontinuous and cocompact action. We say that a Hausdorff compactification $\bar{X}$ of $X$ is perspective if $\varphi$ extends continuously to an action on $\bar{X}$ and $\forall u \in \mathcal{U}, \ \forall K \subseteq X$ compact, $\#\{g \in G: \varphi(g,K)\notin Small(u)\} < \aleph_{0}$, where $\mathcal{U}$ is the only uniform structure compatible with the topology of $\bar{X}$ (given by Theorem 1, $\S 4$, Chapter II of \cite{Bou}).

We denote by $EPers_{0}(\varphi)$ the category whose objects are perspective compactifications of $X$ and morphisms are continuous equivariant maps that are the identity when restricted to $X$. We denote by $EMPers_{0}(\varphi)$ the full subcategory of $EPers_{0}(\varphi)$ whose objects are metrizable spaces. If $G = X$ and the action is the left multiplication action then we use the notation $EPers_{0}(G)$ and $EMPers_{0}(G)$ for such categories.
\end{defi}

\begin{prop}(Proposition 3.1 of \cite{So2}) A compactification $\bar{X}$ of $X$ has the perspectivity property if and only if $\varphi$ extend continuously to $\bar{X}$ and for $K \subseteq X$ a compact subspace, $y \in \bar{X}-X$ and $U$ an open neighbourhood of $y$, then there exists $V$ an open neighbourhood of $y$ such that $V \subseteq U$ and if, for $g \in G$, $\varphi(g,K) \cap V \neq \emptyset$, then $\varphi(g,K) \subseteq U$.
\end{prop}

\begin{defi}\label{pielambda}Let $G$ be a group, $X$ and $Y$ Hausdorff topological spaces with $X$ locally compact and $Y$ compact, $L: G \curvearrowright G$ the left multiplication action, $\varphi: G \curvearrowright X$ a  properly discontinuous cocompact action, $\psi: G \curvearrowright Y$ an action by homeomorphisms and $K \subseteq X$ a compact such that $\varphi(G,K) = X$. Define $\Pi_{K}: Closed(X) \rightarrow  Closed(G)$ as $\Pi_{K}(S) = \{g\in G: \varphi(g,K)\cap S \neq \emptyset\}$ and $\Lambda_{K}: Closed(G) \rightarrow Closed(X)$ as $\Lambda_{K}(F) = \varphi(F,K)$.
\end{defi}

\begin{prop}(Theorem 3.2 and Propositions 3.8 and 3.9 of \cite{So}) Let $K \subseteq X$ be a fundamental domain of the action $\varphi$. The functor $\Pi: EPers_{0}(G) \rightarrow EPers_{0}(\varphi)$ that sends the space $G+_{\partial}Y$ to  $X+_{\partial_{\Pi_{K}}}Y$, the action $L+\psi:G \curvearrowright G+_{\partial}Y$ to $\varphi+\psi: G \curvearrowright X+_{\partial_{\Pi_{K}}}Y$ and the map $id+\phi: G+_{\partial_{1}}Y_{1}\rightarrow G+_{\partial_{2}}Y_{2}$ to $id+\phi: X+_{(\partial_{1})_{ \Pi_{K}}} Y_{1}\rightarrow X+_{(\partial_{2})_{\Pi_{K}}}Y_{2}$, is a isomorphism of categories.

Furthermore, its inverse is the functor $\Lambda: EPers_{0}(\varphi) \rightarrow EPers_{0}(G)$ that sends $X+_{f}Y$ to $G+_{f_{\Lambda_{K}}}Y, \ \varphi+\psi:G \curvearrowright X+_{f}Y$ to $id+\psi: G \curvearrowright G+_{f_{\Lambda_{K}}}Y$ and $id+\phi: X+_{f_{1}}Y_{1}\rightarrow X+_{f_{2}}Y_{2}$ to $id+\phi: G+_{(f_{1})_{\Lambda_{K}}}Y_{1}\rightarrow G+_{(f_{2})_{\Lambda_{K}}}Y_{2}$.
\end{prop}

Since $\Pi$ and $\Lambda$ do not depend of the choice of the fundamental domain (Propositions 3.18 and 3.20 of \cite{So}), we denote $\partial_{\Pi_{K}}$ by $\partial_{\Pi}$ and $f_{\Lambda_{K}}$ by $f_{\Lambda}$.

\begin{cor}(Corollary 3.5 of \cite{So}) Let $G$ be a countable group, $X$ a locally compact Hausdorff space with countable basis and $\varphi: G \curvearrowright X$ properly discontinuous. Then, the functor $\Pi$ maps $EMPers_{0}(G)$ to $EMPers_{0}(\varphi)$ isomorphically.
\end{cor}







\subsection{Convergence actions}

\begin{defi}Let $G$ be a group, $X$ a topological space and $\varphi: G\curvearrowright X$ an action by homeomorphisms. We say that $\varphi$ is properly discontinuous if for every compact set $K \subseteq X$, the set $\{g \in G: \varphi(g,K)\cap K \neq \emptyset\}$ is finite. We say that $\varphi$ is cocompact if the quotient space $X/\varphi$ is compact. A compact subspace $K \subseteq X$ such that $\varphi(G,K) = X$ is called a fundamental domain of $\varphi$.
\end{defi}

\begin{defi}Let $G$ be a group, $X$ a Hausdorff compact topological space and $\varphi: G \curvearrowright X$ an action by homeomorphisms. We say that $\varphi$ has the convergence property if the induced action on the space of distinct triples is properly discontinuous.
\end{defi}

\begin{obs}Regardless the definition above allow us to consider these actions, we are not considering $\varphi$ as a convergence action if the set $X$ has cardinality $2$, unless that $G$ is virtually cyclic, $X = Ends(G)$ and $\varphi$ is the action induced by the left multiplication action $G \curvearrowright G$. \end{obs}

\begin{defi}\label{attractorsumcomp}Let $G$ be a group acting properly discontinuously on a locally compact Hausdorff space $X$ and acting on a Hausdorff compact space $Y$ with the convergence property. The attractor-sum compactification of $X$ is the unique compactification of $X$ with $Y$ as its remainder and such that the action of $G$ on it (that extends both actions) has the convergence property. We denote such compactification by $X+_{f_{c}} Y$ (where $c$ means that the action still has the convergence property).
\end{defi}

\begin{obs}The existence and uniqueness of the attractor-sum compactification is due to Gerasimov (Proposition 8.3.1 of \cite{Ge2}).
\end{obs}

\begin{prop}\label{convergenceisperspective}(Gerasimov - Proposition 7.5.4 of \cite{Ge2}) Let $G$ be a group acting on a Hausdorff compact space $Y$ with the convergence property. Then the compactification $G+_{\partial_{c}}Y$ has the group theoretic perspectivity property.
\end{prop}

\begin{defi}Let $\varphi: G \curvearrowright X$ be a convergence action and $p \in X$. We say that $p$ is a conical point if there is an infinite set $K \subseteq G$ such that $\forall q \neq p$, $Cl_{X}(\{(\varphi(g,p),\varphi(g,q)): g \in K\})\cap \Delta X = \emptyset$.
\end{defi}

\begin{defi}Let $\varphi: G \curvearrowright X$ be an action by homeomorphisms. A point $p \in X$ is bounded parabolic if the action $\varphi|_{Stab_{\varphi}p\times X-\{p\}}: Stab_{\varphi}p \curvearrowright X-\{p\}$ is properly discontinuous and cocompact.
\end{defi}

\begin{defi}\label{relhyppair}Let $G$ be a group, $X$ a Hausdorff compact space and $\varphi: G \curvearrowright X$ a minimal action by homeomorphisms. We say that $\varphi$ is relatively hyperbolic if it has the convergence property and its limit set has only conical and bounded parabolic points. If $\mathcal{P}$ is a representative set of conjugation classes of stabilizers of bounded parabolic points of $X$, then we say that $G$ is relatively hyperbolic with respect to $\mathcal{P}$ (or equivalently that $(G,\mathcal{P})$ is a relatively hyperbolic pair). We say that $X$ is the Bowditch boundary of the pair $(G,\mathcal{P})$ and we denote it by $X = \partial_{B}(G,\mathcal{P})$.

If $(G,\mathcal{P})$ is a relatively hyperbolic pair, then we call $G+_{\partial_{c}}\partial_{B}(G,\mathcal{P})$ a relatively hyperbolic compactification of $G$.
\end{defi}

\begin{obs}If $(G,\mathcal{P})$ is a relatively hyperbolic pair, then its Bowditch boundary is uniquely defined, up to unique equivariant homeomorphism (Corollary 6.1(e) of \cite{GP2}).

An equivalent definition is that the minimal convergence action on $X$ induces a cocompact action on the space of distinct pairs (1C of \cite{Tu} and Main Theorem of \cite{Ge1}).
\end{obs}

\subsection{Floyd compactification}

Let $\Gamma$ be a locally finite graph, $(X,d)$ its set of vertices with the metric induced by the metric on the graph and $f: \N \rightarrow \R_{>0}$ a map satisfying:

\begin{enumerate}
    \item $\exists k > 0$ such that $\forall n \in \N$, $1 \leqslant \frac{f(n)}{f(n+1)} \leqslant k$.   
    \item The series $\sum_{n \in \N}f(n)$ converges.
\end{enumerate}

If $v \in X$, then we get a new metric on $X$ given by $\delta_{v}(x,y) = \\ \inf\{\sum_{i =1}^{k} f(d(v,\{x_{i},x_{i+1}\})): x = x_1,...,x_{k} = y$ is a path between $x$ and $y\}$. The Floyd compactification of $X$ with respect to the Floyd map $f$ (i.e. a map that satisfies the conditions above) and base point $v$ is the Cauchy completion of $X$ with respect to the metric $\delta_{v}$ (as a topological space, it does not depend of the choice of the base point). We will denote it by $X+_{\partial_{f}}\partial_{f}(X)$. See \cite{Fl} for details.

\begin{prop}The Floyd compactification $X+_{\partial_{f}}\partial_{f}(X)$ is a perspective compactification of $(X,\varepsilon_{d})$.
\end{prop}

\begin{proof}By Proposition 2.2 of \cite{GP}, the Floyd compactification of a locally finite graph has the Karlsson property, which implies, by Proposition 3.41 of \cite{So}, that it has the perspectivity property.
\end{proof}

\begin{obs}Let $\U_{f}$ be the uniform structure compatible with the topology of $X+_{\partial_{f}}\partial_{f}(X)$. By the Karlsson property, we mean that for every $u \in \U_{f}$, there exists a bounded set $S \subseteq X$ (i.e. $S$ is finite) such that every geodesic that does not intersect $S$ is $u$-small.

For finitely generated  groups, this proposition is an immediate consequence of the fact that groups act on their Floyd boundaries with the convergence property (Gerasimov-Potyagailo, Proposition 4.3.1 of \cite{GP}) and the fact that equivariant compactifications of a group with the convergence property have the perspectivity property (Proposition 7.5.4 of \cite{Ge2}).
\end{obs}

\section{Coarse geometric perspectivity vs group theoretic perspectivity}
\label{cgeoperspvpersp}

\begin{defi}\label{coarseactionstructure}Let $G$ be a group, $X$ a Hausdorff locally compact paracompact space and $\varphi: G \curvearrowright X$ a properly discontinuous and cocompact action. For $A \subseteq X$ we define the saturation of $A$ as $Sat(A) = \{(\varphi(g,x),\varphi(g,x')): g \in G, x,x' \in A\}$. We denote by $\varepsilon_{\varphi}$ the coarse structure on $X$ generated by the set $\{Sat(A): A$ is topologically bounded on $X\}$. On the special case where $X = G$ and $\varphi = L$ is the left multiplication action, we denote $\varepsilon_{\varphi}$ by $\varepsilon_{G}$.
\end{defi}

\begin{prop}Let $G$ be a group, $X$ a Hausdorff locally compact paracompact space, $\varphi: G \curvearrowright X$ a properly discontinuous and cocompact action. Then $\mathcal{A} = \{Sat(U_{1})\circ ...\circ Sat(U_{n}): U_{1},...,U_{n}$ are topologically bounded $\}$ is a basis for the coarse structure $\varepsilon_{\varphi}$.
\end{prop}

\begin{proof}By the definition, $\mathcal{A}$ is closed under compositions. We have also that $(Sat(U_{1})\circ ...\circ Sat(U_{n}))^{-1} = Sat(U_{n})^{-1}\circ ...\circ Sat(U_{1})^{-1} = Sat(U_{n})\circ ...\circ Sat(U_{1}) \in \varepsilon_{\varphi}$. 

Since $\varphi$ is cocompact, there exists a fundamental domain $K$. Then $\Delta X \subseteq Sat(K)$.

Let $A,B \in \mathcal{A}$. So $A = Sat(U_{1})\circ ...\circ Sat(U_{n})$ and $B = Sat(V_{1})\circ ...\circ Sat(V_{m})$, for some $U_{1},...,U_{n},V_{1},...,V_{m}$ topologically bounded sets. Then $A \cup B \subseteq  Sat(U_{1}\cup...\cup U_{n}\cup V_{1}\cup...\cup V_{m})\circ ... \circ Sat(U_{1}\cup...\cup U_{n}\cup V_{1}\cup...\cup V_{m}) \in \mathcal{A}$ (product of $\max\{n,m\}$ copies of $Sat(U_{1}\cup...\cup U_{n}\cup V_{1}\cup...\cup V_{m})$).

Thus $\mathcal{A}$ is a basis for $\varepsilon_{\varphi}$.
\end{proof}

\begin{lema}Let $\varphi: G \curvearrowright X$ be a properly discontinuous action. Then for every finite family $B_{1},..., B_{n+1}$ of topologically bounded sets, the set $\{(g_{1},...,g_{n}) \in G^{n}: \forall i \in \{1,...,n-1\}, \varphi(g_{i},B_{i}) \cap \varphi(g_{i+1},B_{i+1}) \neq \emptyset, \varphi(g_{n},B_{n}) \cap B_{n+1} \neq \emptyset\}$ is finite.
\end{lema}

\begin{proof}If $n = 1$, this is just the definition of properly discontinuous actions.

Suppose the proposition holds for $n = k$. Let $B_{1},..., B_{k+2}$ be bounded sets and $A \! = \! \{(g_{1},...,g_{k+1}) \! \in G^{k+1}\!:\! \forall i \in \{1,...,k\}, \varphi(g_{i},B_{i}) \cap \varphi(g_{i+1},B_{i+1}) \neq \emptyset,$ $\varphi(g_{k+1},B_{k+1}) \cap B_{k+2} \neq \emptyset\}$ and $A_{i}$ the projection of $A$ to the $i$-th coordinate. We have that $A \!=\! \bigcup_{g_{k+1} \in A_{k+1}}  \{(g_{1},...,g_{k}) \!\in G^{k}:\! \forall i \!\in\! \{1,...,k-1\},$ $\varphi(g_{i},B_{i}) \cap \varphi(g_{i+1},B_{i+1}) \neq \emptyset, \varphi(g_{k},B_{k}) \cap \varphi(g_{k+1},B_{k+1}) \neq \emptyset\} \times \{g_{k+1}\}$. The choices of $g_{k+1}$ are finite since $\varphi$ is properly discontinuous and each factor of this union is finite because of the induction hypothesis. Thus $A$ is finite.
\end{proof}

\begin{prop}Let $G$ be a group, $X$ a Hausdorff locally compact paracompact space and $\varphi: G \curvearrowright X$ a properly discontinuous and cocompact action. Then $(X,\varepsilon_{\varphi})$ is a proper and coarsely connected coarse space.
\end{prop}

\begin{proof}Let $x,y \in X$. Then $(x,y) \in Sat(\{x,y\})$, which implies that $X$ is coarsely connected.

Since $\varphi$ is cocompact and $X$ is locally compact, there is an open set $U$ that is topologically bounded and contains a fundamental domain $K$. Let $x \in X$. there exists $x' \in U$ and $g \in G$ such that $\varphi(g,x') = x$. Then $(x,x) = (\varphi(g,x'),\varphi(g,x')) \in Sat(U)$. We also have that $Sat(U)$ is open, since it is the union of the open sets of the form $(\varphi(g,\_)\times \varphi(g,\_))(U\times U)$. Thus $Sat(U)$ is a neighbourhood of $\Delta X$ that is an element of $\varepsilon_{\varphi}$.

Let $B$ be a bounded set in $X$. Then there exists $b \in X$ such that $B\times \{b\} \in \varepsilon_{\varphi}$. We have also that there exists $B_{1},...,B_{n}$ topologically bounded subsets of $X$ such that $B\times \{b\} \subseteq Sat(B_{n}) \circ...\circ Sat(B_{1})$. Since $B_{1},...B_{n}$ are topologically bounded and $\varphi$ is properly discontinuous, the set $A = \{(g_{1},...,g_{n}) \in G^{n}: \forall i \in \{1,...,n-1\}, \varphi(g_{i},B_{i}) \cap \varphi(g_{i+1},B_{i+1}) \neq \emptyset, b \in \varphi(g_{n},B_{n})\}$ is finite. If $(a,b) \in B\times \{b\}$, then there exists $(g_{1},g_{2},...,g_{n}) \in G^{n}$ and $x_{i},y_{i} \in B_{i}$ such that $\forall i \in \{1,...,n\}$, $(\varphi(g_{i},x_{i} ),\varphi(g_{i},y_{i})) \in Sat(B_{i})$, $a = \varphi(g_{1},x_{1})$, $\forall i \in \{1,...,n-1\}$, $\varphi(g_{i},y_{i}) = \varphi(g_{i+1},x_{i+1})$ and $\varphi(g_{n},y_{n}) = b$. So for every index $i \in \{1,...,n-1\}$, $\varphi(g_{i},B_{i}) \cap \varphi(g_{i+1},B_{i+1}) \neq \emptyset$ and $b \in \varphi(g_{n},B_{n})$, which implies that $(g_{1},g_{2},...,g_{n}) \in A$. Then $B \subseteq  \bigcup\{\varphi(g_{1},B_{1}): \exists (g_{1},g_{2},...,g_{n}) \in A\}$, which is topologically bounded since it is a finite union of topologically bounded sets.

Thus $(X,\varepsilon_{\varphi})$ is proper.
\end{proof}

\begin{prop}\label{milnorsvarccoarse}Let $G$ be a group, $X$ a Hausdorff locally compact paracompact space and $\varphi: G \curvearrowright X$ a properly discontinuous action. If $x_{0} \in X$, then the map $\varphi_{x_{0}}: (G, \varepsilon_{G}) \rightarrow (X, \varepsilon_{\varphi})$ such that $\varphi_{x_{0}}(g) = \varphi(g,x_{0})$ is a coarse embedding. Moreover, if $\varphi$ is cocompact, then $\varphi_{x_{0}}$ is a coarse equivalence.
\end{prop}

\begin{obs} This is a version of the Milnor-$\check{S}$varc Lemma for coarse spaces.
\end{obs}

\begin{proof}Let $S$ be a bounded  set on the space $X$. We have that the set $\varphi_{x_{0}}^{-1}(S) = \{g \in G: \varphi(g,x_{0}) \in S\}$ is finite because $\varphi$ is properly discontinuous. Then $\varphi_{x_{0}}$ is proper. 

Let $e \in \varepsilon_{G}$. Let, for $g \in G$, $\Delta_{g} = \{(h,hg): h \in G\}$ and suppose that $e = \Delta_{g}$. Then $\varphi_{x_{0}}(e) = \{(\varphi(h,x_{0}),\varphi(h,\varphi(g,x_{0}))): h \in G\} \subseteq Sat(\{x_{0},\varphi(g,x_{0})\}) \in \varepsilon_{\varphi}$. Suppose that $e = Sat(U)$, with $U$ a finite subset of $G$. Then $Sat(U) = \{(gh,gh'): g \in G, h,h' \in U\} = \{(g',g'h^{-1}h'): g' \in G, h,h' \in U\} = \bigcup_{h,h' \in U} \Delta_{h^{-1}h'}$, which is a finite union since $U$ is finite. Then $\varphi_{x_{0}}(Sat(U)) = \bigcup_{h,h' \in U} \varphi_{x_{0}}(\Delta_{h^{-1}h'}) \subseteq Sat(\{x_{0}, \varphi(h^{-1}h',x_{0}): h,h' \in U\}) \in \varepsilon_{\varphi}$. By \textbf{Propositions \ref{coarsecomposition}} and \textbf{\ref{coarsebasis}}, it follows that $\varphi_{x_{0}}$ is bornologous.

Let $f: Orb_{\varphi}x_{0} \rightarrow G$ be a map such that $f(x)$ chooses an element in $\varphi^{-1}_{x_{0}}(x)$. 
Let $U$ be a finite set of $G$. Then $f^{-1}(U) \subseteq \varphi^{-1}_{x_{0}}(U)$, which is finite because $\varphi$ is properly discontinuous.

Let $e = Sat(U)$, where $U$ is a topologically bounded subset of $Orb_{\varphi}x_{0}$ (this means that $U$ is finite). Let $S = \{s \in G:  \varphi(s,x_{0}) \in U\}$. Since $U$ is finite and $\varphi$ is properly discontinuous, then $S$ is also finite. Then we have:

$$f(e) = \{(f(\varphi(g,x)),f(\varphi(g,x'))): g \in G, x,x' \in U\} \subseteq$$ $$\bigcup_{x,x' \in U} \bigcup_{g \in G}\{\varphi_{x_{0}}^{-1}(\varphi(g,x))\times \varphi_{x_{0}}^{-1}(\varphi(g,x'))\} =$$ 
$$\bigcup_{s,s' \in S} \bigcup_{g \in G}\{\varphi_{x_{0}}^{-1}(\varphi(gs,x_{0}))\times \varphi_{x_{0}}^{-1}(\varphi(gs',x_{0}))\} =$$ 
$$\bigcup_{s,s' \in S} \bigcup_{g \in G} (gs Stab \ x_{0})\times (gs'Stab \ x_{0}) = \bigcup_{s,s' \in S} \bigcup_{g \in G} \bigcup_{h \in Stab \ x_{0}} \{gs h\}\times \{gs'h\} =$$ 
$$\bigcup_{s,s' \in S} \bigcup_{g' \in G} \bigcup_{h \in Stab \ x_{0}} \{g'\}\times \{g'h^{-1}s^{-1}s'h\} =  \bigcup_{s,s' \in S} \bigcup_{h \in Stab \ x_{0}} \Delta_{h^{-1}s^{-1}s'h}$$

Which is an element of $\varepsilon_{G}$, since this is a finite union. So $f$ is bornologous.

We have that $\varphi_{x_{0}} \circ f = id_{Orb_{\varphi}x_{0}}$, which implies that $id_{Orb_{\varphi}x_{0}}$ is close to $\varphi_{x_{0}} \circ f$, and $\{(g,f(\varphi_{x_{0}}(g))): g \in G\} \subseteq \{(g,\varphi_{x_{0}}^{-1}(\varphi_{x_{0}}(g))): g \in G\} = \{(g,g\varphi_{x_{0}}^{-1}(\varphi_{x_{0}}(1))): g \in G\} = \bigcup_{h \in \varphi_{x_{0}}^{-1}(\varphi_{x_{0}}(1))} \Delta_{h} \in \varepsilon_{G}$, which implies that $id_{G}$ is close to $f \circ \varphi_{x_{0}}$. Thus $\varphi_{x_{0}}$ is a coarse embedding.

Suppose that $\varphi$ is cocompact and let $K \subseteq X$ be a fundamental domain. We have that $\B(Orb_{\varphi}x_{0},Sat(K)) = \{a \in X: \exists b \in Orb_{\varphi}x_{0}: (a,b) \in Sat(K)\} = \{a \in X: \exists b \in Orb_{\varphi}x_{0}: (a,b) = (\varphi(g,x),\varphi(g,x')): g \in G, x,x' \in K\} = \{\varphi(g,x) \in X: x \in K, g \in G, \exists x' \in K: \varphi(g,x') \in Orb_{\varphi}x_{0}\} = \{\varphi(g,x) \in X: x \in K, g \in G\} = X$. Then $Orb_{\varphi}x_{0}$ is $Sat(K)$-quasi-dense, which implies that $\varphi_{x_{0}}$ is a coarse equivalence by \textbf{Corollary \ref{quasidense}}.
\end{proof}

\begin{prop}\label{coarsegroup} Let $G$ be a group, $X$ a Hausdorff locally compact paracompact space, $\varphi: G \curvearrowright X$ a properly discontinuous and cocompact action and $X+_{f}W \in Comp(X)$ such that $\varphi$ extends continuously to an action on $X+_{f}W$. Then $X+_{f}W \in EPers_{0}(\varphi)$ if and only if $X+_{f}W \in Pers(\varepsilon_{\varphi})$.
\end{prop}

\begin{obs}In another words: $X+_{f}W$ has the group theoretic perspectivity property with respect to the action $\varphi$ if and only if it has the coarse geometric perspectivity property with respect to the coarse structure $\varepsilon_{\varphi}$.
\end{obs}

\begin{proof}$(\Rightarrow)$ Since every element of $\varepsilon_{\varphi}$ is proper and $\varepsilon_{\varphi}$ is generated by the set of saturations of compact subsets of $X$, it is sufficient to prove that every saturation of a compact set is in $\varepsilon_{f}$. Let $K$ be a compact subset of $X$, $y \in W$ and $V$ a neighbourhood of $y$. There is a neighbourhood $U$ of $y$ such that $U \subseteq V$ and if $g \in G$ such that $\varphi(g,K) \cap U \neq \emptyset$, then $\varphi(g,K) \subseteq V$. So the set $Sat(K) \cap (U\times (X-V)) = \{(\varphi(g,x),\varphi(g,x')): g \in G, x,x' \in K, \varphi(g,x) \in U, \varphi(g,x) \in X-V \} = \emptyset$, by the definition of $U$. Thus $Sat(K) \in \varepsilon_{f}$.

$(\Leftarrow)$ Now we have that $\varepsilon_{\varphi} \subseteq \varepsilon_{f}$. Let $K$ be a compact subset of $X$, $y \in W$ and $V$ a neighbourhood of $y$. Since $Sat(K) \in \varepsilon_{f}$, there exists a neighbourhood $U$ of $y$ such that $U \subseteq V$ and $Sat(K) \cap (U\times (X-V)) = \emptyset$. Let $g \in G$ such that $\varphi(g,K) \cap U \neq \emptyset$. We have that there is no pair $(\varphi(g,x),\varphi(g,x')) \in U\times (X-V)$ such that $x,x' \in K$, which implies that $\varphi(g,K) \subseteq V$. Thus $X+_{f}W \in EPers_{0}(\varphi)$.
\end{proof}

Then $EPers_{0}(\varphi)$ is the subcategory of $Pers(\varepsilon_{\varphi})$ whose objects are spaces where the action $\varphi$ extends continuously to the whole space and the morphisms are equivariant maps.

\begin{prop}\label{differentperspectivities}Suppose that $G$ is a countable group and $X$ has a countable basis. Let $\Lambda: EMPers_{0}(\varphi) \rightarrow EMPers_{0}(G)$ be the functor given by the action $\varphi$ and $\Lambda': MPers(\varepsilon_{\varphi}) \rightarrow MPers(\varepsilon_{G})$ be the functor defined by the pullback of $\varphi_{x_{0}}$, for any $x_{0} \in X$. Then $\Lambda'|_{EMPers_{0}(\varphi)} = \Lambda$.
\end{prop}

\begin{proof}Let $K \subseteq X$ be a fundamental domain of $\varphi$. Let $X+_{f}W \in EMPers_{0}(\varphi)$. We have that $\Lambda(X+_{f}W) = G+_{f_{\Lambda}}W$, where $f_{\Lambda}(F) = f(\varphi(F,K))$  and $\Lambda'(X+_{f}W) = G+_{f^{\ast}}W$, where $f^{\ast}(F) = f(Cl_{X}(\varphi_{x_{0}}(F))) = f(\varphi(F,x_{0}))$. So $\forall F \in Closed(G)$, $f^{\ast}(F) \subseteq f_{\Lambda}(F)$, which implies, by \textbf{Corollary \ref{continuousidentity}}, that the identity map $id: G+_{f^{\ast}}W \rightarrow G+_{f_{\Lambda}}W$ is continuous and then a homeomorphism, since both spaces are compact and Hausdorff. Then  $f^{\ast} = f_{\Lambda}$, which implies that  $\Lambda(X+_{f}W) =  \Lambda'(X+_{f}W)$. It is clear that $\Lambda$ and $\Lambda'$ are equal on the morphisms. Thus $\Lambda'|_{EMPers_{0}(\varphi)} = \Lambda$.
\end{proof}

Let $\psi_{x_{0}}: (X,\varepsilon_{\varphi}) \rightarrow(G,\varepsilon_{G})$ be a coarse inverse of $\varphi_{x_{0}}$.

\begin{teo}\label{main}Let $\Pi: EMPers_{0}(G) \rightarrow EMPers_{0}(\varphi)$ be the functor given by the action $\varphi$ and $\Pi': MPers(\varepsilon_{G}) \rightarrow MPers(\varepsilon_{\varphi})$ be the functor defined by the pullback of $\psi_{x_{0}}$. Then $\Pi'|_{EMPers_{0}(G)} = \Pi$.
\end{teo}

\begin{proof}It is immediate from the fact that the pairs $\Pi,\Lambda$ and $\Pi', \Lambda'$ are inverses (the compositions are exactly the identity functors and not just naturally homeomorphic to the identity functors).
\end{proof}

Since the construction of the functors given by coarse equivalences agree with Guilbault and Moran's construction (Corollary 7.3 of \cite{GM}) for $\mathcal{Z}$-structures and the functors given by actions agree with the functors given by coarse equivalences for equivariant compactifications, we have that compositions of functors given by actions agree with Guilbault and Moran's construction for $E\mathcal{Z}$-structures. Precisely we have:

\begin{cor}\label{ezcorrespondence}Let $G$ be a group and $\varphi_{i}: G \curvearrowright X_{i}$ properly discontinuous and cocompact actions, $i = \{1,2\}$, where $X_{1}$ and $X_{2}$ are Hausdorff locally compact metric ANR spaces. If $\Pi_{i}: EMPers_{0}(G) \rightarrow EMPers(\varphi_{i})$ and $\Lambda_{i}: EMPers_{0}(\varphi_{i}) \rightarrow EMPers_{0}(G)$ are the  functors given by the action $\varphi_{i}$, then $\Pi_{2} \circ \Lambda_{1}$ and $\Pi_{1}\circ \Lambda_{2}$ preserve $E\mathcal{Z}$-structures. \eod
\end{cor}






\section{Limits}

\begin{prop}Let $(X,\varepsilon)$ be a locally compact paracompact Hausdorff
space with a proper coarsely connected coarse structure. Then $Pers(\varepsilon)$ is closed under small limits.
\end{prop}

\begin{proof}Let's consider $\mathcal{C}$ a small category, $F: \mathcal{C} \rightarrow Pers(\varepsilon)$ a functor and $\tilde{F}: \mathcal{C} \rightarrow Comp(X)$ the functor that does the same as $F$. We have that $\lim\limits_{\longleftarrow} \tilde{F}$ exists. It is easy to see that if $\lim\limits_{\longleftarrow} \tilde{F}$ has the perspectivity property, then $\lim\limits_{\longleftarrow} \tilde{F} = \lim\limits_{\longleftarrow} F$.

Let $\tilde{X}$ be the maximal perspective compactification of $X$ (its existence is given by Proposition 2.39 of \cite{Ro}). Then, for every object $c \in \mathcal{C}$, there is a unique morphism $\psi_{c}: \tilde{X} \rightarrow F(c)$. The family $\{\psi_{c}\}_{c \in \mathcal{C}}$ induces a morphism in $\mathcal{C}$ $\psi: \tilde{X} \rightarrow \lim\limits_{\longleftarrow} \tilde{F}$. By \textbf{Proposition \ref{quotientpersp}}, it follows that $\lim\limits_{\longleftarrow} \tilde{F}$ has the perspectivity property.
\end{proof}

\section{$\delta$-Hyperbolic spaces}

\subsection{Boundaries that are accessible by coarse arcs}

\begin{defi}Let $(X,\varepsilon)$ be a locally compact paracompact Hausdorff
space with a proper coarsely connected coarse structure and $\Psi$ a set of proper injective maps of the form $\gamma: A_{\gamma} \rightarrow X$, where $A_{\gamma}$ is a closed unbounded subset of $[0,\infty)$ that contains $0$ and there exists $p \in X$ such that $\forall \gamma \in \Psi$, $\gamma(0) = p$. A perspective compactification $X+_{f}W$ is $\Psi$-accessible if the following conditions are satisfied:

\begin{enumerate}
    \item For every $w \in W$, there exists $\gamma \in \Psi$ such that $f(Cl_{X}(Im \ \gamma)) = \{w\}$. 
    \item If $\gamma \in \Psi$, then $\# f(Cl_{X}(Im \ \gamma)) = 1$.
\end{enumerate}
\end{defi}

\begin{obs}Let $A$ be a closed unbounded subset of $[0,\infty)$, $X$ a Hausdorff locally compact space and $\gamma: A \rightarrow X$ a continuous proper map. Then $\# f(Im \ \gamma) = 1$ if and only if the map  $\gamma$ has a continuous extension to a map $\bar{\gamma}: A \cup\{\infty\} \rightarrow X+_{f}W$, where  $A \cup\{\infty\}$ has the subspace topology given by $[0,\infty]$. In this case $\bar{\gamma}(\infty)$ is the only point in $f(Im \ \gamma)$. 
\end{obs}

\begin{prop}\label{quotientaccess}Let $(X,\varepsilon)$ be a locally compact paracompact Hausdorff space with a proper coarsely connected coarse structure, let $X+_{f}W$ and $X+_{g}Z$ two perspective compactifications of $X$ and a continuous map $id+\phi: X+_{f}W \rightarrow X+_{g}Z$. If $X+_{f}W$ is $\Psi$-accessible, and the maps in $\Psi$ are continuous, then $X+_{g}Z$ is $\Psi$-accessible.
\end{prop}

\begin{proof}Let $z \in Z$ and $w \in W$ such that $\phi(w) = z$. Since the space $X+_{f}W$ is $\Psi$-accessible, then there exists a map $\gamma \in \Psi$ such that its extension $\bar{\gamma}: Dom(\gamma) \cup\{\infty\} \rightarrow X+_{f}W$ such that $\bar{\gamma}(\infty) = w$ is continuous at $w$. Since $id+\phi$ is continuous, then the extension $\bar{\gamma}': Dom(\gamma) \cup\{\infty\} \rightarrow X+_{g}Z$ given by $\bar{\gamma}'(\infty) = \phi(w) = z$ is continuous at $\infty$.

Let $\gamma \in \Psi$. Then there exists a point $w \in W$ such that the extension $\bar{\gamma}: Dom(\gamma) \cup\{\infty\} \rightarrow X+_{f}W$ given by $\bar{\gamma}(\infty) = w$ is continuous at $w$. Then the extension $\bar{\gamma}': Dom(\gamma) \cup\{\infty\} \rightarrow X+_{g}Z$ given by $\bar{\gamma}'(\infty) = \phi(w)$ is continuous at $\infty$. Thus $X+_{g}Z$ is $\Psi$-accessible.
\end{proof}

\begin{defi}Let $(X,\varepsilon)$ be a locally compact paracompact Hausdorff
space with a proper coarsely connected coarse structure and $\Psi$ a set of proper continuous injective maps. We say that $\Psi$ is closed for limits if all maps in $\Psi$ have the same domain and if a net of elements in $\Psi$ converges uniformly to a map, then this map is in $\Psi$.  
\end{defi}

\begin{ex}If $X$ is a proper geodesic metric space, then the set of geodesic rays starting at a fixed point $x \in X$ is closed for limits (consequence of Proposition 2.3.1 of \cite{Pa}).
\end{ex}





\begin{prop}\label{pushforwardaccess}Let $(X,\varepsilon)$ and $(Y,\eta)$ be  locally compact paracompact Hausdorff spaces with proper coarsely connected coarse structures and $\pi: Y \rightarrow X$ be a coarse equivalence with coarse inverse $\varpi$. If $X+_{f}W$ is a perspective compactification that is $\Psi$-accessible, then $\Pi(X+_{f}W)$ is $\varpi(\Psi)$-accessible, where $\varpi(\Psi) = \{\varpi \circ \gamma: \gamma \in \Psi\}$ and $\Pi$ is the functor induced by $\pi$.
\end{prop}

\begin{proof}We have that $\Pi(X+_{f}W) = Y+_{f^{\ast}}W$, where $f^{\ast}$ is the pullback of $f$ with respect to $\pi$ and $id_{W}$. 

Let $w \in W$. There exists $\gamma \in \Psi$ such that $f(Cl_{X}(Im \ \gamma)) = \{w\}$.  Then $\varpi \circ \gamma \in \varpi(\Psi)$ and $f^{\ast}(Cl_{Y}(Im \ \varpi \circ \gamma))  = f(Cl_{X}(\pi(Cl_{Y}(Im \ \varpi \circ \gamma)))) = f(Cl_{X}(Im \ \pi \circ \varpi \circ \gamma)) = f(Cl_{X}(Im \ \gamma)) = \{w\}$, by \textbf{Proposition \ref{closesameboundary}}.

Let  $\varpi \circ \gamma \in \varpi(\Psi)$. Then $\gamma \in \Psi$, which implies that $\# f(Cl_{X}(Im \ \gamma)) = 1$. But we have that $f(Cl_{X}(Im \ \gamma)) = f^{\ast}(Cl_{Y}(Im \ \varpi \circ \gamma))$, which implies that $\# f^{\ast}(Cl_{Y}(Im \ \varpi \circ \gamma)) = 1$.

Thus $\Pi(X+_{f}W)$ is $\varpi(\Psi)$-accessible.
\end{proof}

\subsection{$\delta$-Hyperbolic spaces}

\begin{defi}Let $(X,d)$ be a proper metric space and $p \in X$. Let $\Phi_{p}$ be the set of all geodesic rays starting at $p$.
\end{defi}

\begin{ex}Let $(X,d)$ be a proper $\delta$-hyperbolic geodesic space (see \cite{BH} for the definition)  and $p \in X$. Then the hyperbolic compactification $X+_{f_{\infty}}\partial_{\infty}(X)$ is a perspective compactification of $(X, \varepsilon_{d})$ (it follows from Lemma 6.23 of \cite{Ro} and \textbf{Proposition \ref{quotientpersp}}) and $\Phi_{p}$-accessible (by the usual construction of the hyperbolic compactification).

\end{ex}

\begin{prop}Let $(X,d)$ be a proper $\delta$-hyperbolic geodesic space and $p\in X$. Then the hyperbolic compactification of $X$ is the limit of all $\Phi_{p}$-accessible compactifications of $(X,\varepsilon_{d})$. 
\end{prop}


\begin{proof}Let $X+_{f}W$ be the limit of all $\Phi_{p}$-accessible compactifications of $X$. So there is a continuous surjective map $id+\phi: X+_{f}W \rightarrow X+_{f_{\infty}}\partial_{\infty}(X)$. Let $X+_{g}Z$ be a compactification of $X$ such that it is $\Phi_{p}$-accessible. There is also a continuous surjective map $id+\phi': X+_{f}W \rightarrow X+_{g}Z$.

Let $x \in \partial_{\infty}(X)$. For a geodesic ray $\gamma$ that starts at $p$ such that $f_{\infty}(Im \ \gamma) = \{x\}$, we define $\psi(x) = g(Im \ \gamma)$, which is a single point since $X+_{g}Z$ is $\Phi_{p}$-accessible. This defines a map $\psi:  \partial_{\infty}(X) \rightarrow Z$. It is well defined, since for any other geodesic ray $\gamma'$ that starts at $p$ and such that $f_{\infty}(Im \ \gamma') = \{x\}$, $Im \ \gamma$ and $Im \ \gamma'$ have finite Hausdorff distance (by the definition of the hyperbolic boundary), which implies that $g(Im \ \gamma) = g(Im \ \gamma')$ (\textbf{Proposition \ref{closesameboundary}}). 

Let $w \in W$. There exists a geodesic ray $\gamma$ that starts at $p$ and such that $f(Im \ \gamma) = \{w\}$. Since $id+\phi$ and $id+\phi'$ are continuous maps, then $\{\phi(w)\} = f_{\infty}(Im \ \gamma)$ and $\{\phi'(w)\} = g(Im \ \gamma)$. So the following diagram commutes:

$$ \xymatrix{  X+_{f}W \ar[r]_{id+\phi} \ar[rd]_{id+\phi'} & X+_{f_{\infty}}f_{\infty}(X) \ar[d]_{id+\psi} \\
                & X+_{g}Z } $$

Since $id+\phi$ is a quotient map and  $id+\phi'$ is continuous, then the map $id+\psi:  X+_{f_{\infty}}f_{\infty}(X) \rightarrow X+_{g}Z$ is continuous. 

So, for every compactification $X+_{g}Z$ that it is $\Phi_{p}$-accessible, there is a morphism $X+_{f_{\infty}}f_{\infty}(X) \rightarrow X+_{g}Z$ (which is unique since $X$ is a dense subset of the compactification), which implies that $X+_{f_{\infty}}f_{\infty}(X)$ is the limit of all $\Phi_{p}$-accessible compactifications of $X$.
\end{proof}

\begin{prop}\label{changebasepoint}Let $(X,d)$ be a proper $\delta$-hyperbolic geodesic space and $p\in X$. If $X+_{f}W$ is a perspective compactification of $X$ that is $\Phi_{p}$-accessible, then it is $\Phi_{q}$-accessible for every $q \in X$. of $(X,\varepsilon_{d})$. 
\end{prop}

\begin{proof}Since the space $X+_{f}W$ is $\Phi_{p}$-accessible, then there is a quotient map $id+\phi: X+_{f_{\infty}}\partial_{\infty}(X) \rightarrow X+_{f}W$. Since $X+_{f_{\infty}}\partial_{\infty}(X)$ is $\Phi_{q}$-accessible $\forall q \in X$, we have that $X+_{f}W$ is $\Phi_{q}$-accessible $\forall q \in X$ (by \textbf{Proposition \ref{quotientaccess}}).
\end{proof}

\begin{prop}Let $(X,d)$ and $(Y,d')$ be proper $\delta$-hyperbolic geodesic spaces and $\pi: Y \rightarrow X$ a quasi-isometry with quasi-inverse $\varpi$. Then the functor $\Pi$ sends $\Phi_{p}$-accessible compactifications to $\Phi_{\varpi(p)}$-accessible compactifications.
\end{prop}

\begin{proof}Let $X+_{f}W$ be a perspective compactification of $X$ that is $\Phi_{p}$-accessible. We have that $\Pi(X+_{f}W) = Y+_{f^{\ast}}W$. By \textbf{Proposition \ref{pushforwardaccess}}, the compactification $Y+_{f^{\ast}}W$ is $\varpi(\Phi_{p})$-accessible. 

Let $\gamma \in \Phi_{\varpi(p)}$. Then $\pi \circ \gamma$ is a quasi-geodesic ray, which implies, by Lemma III.3.1 of \cite{BH}, that there exists a geodesic ray $\lambda \in \Phi_{p}$ such that $Im \ \lambda$ and $Im \  \pi \circ \gamma$ have finite Hausdorff distance. Then $Im \ \varpi \circ \lambda$ and $Im \ \varpi \circ \pi \circ \gamma$ have finite Hausdorff distance. However, $Im \ \varpi \circ \pi \circ \gamma$ and $Im \ \gamma$ have finite Hausdorff distance, since $\pi$ and $\varpi$ are quasi-inverses. Then $Im \ \gamma$ and $Im \ \varpi \circ \lambda$  have finite Hausdorff distance, which implies that $f^{\ast}(Im \ \gamma) = f^{\ast}(Cl_{Y}(Im \ \varpi \circ \lambda)) = f(\lambda)$, which implies that $\# f^{\ast}(Im \ \gamma) = 1$.

Let $x \in W$. There exists $\gamma \in \Phi_{p}$ such that $f(Im \ \gamma) = \{x\}$. Then $\varpi \circ \gamma$  is a quasi-geodesic such that its image has finite Hausdorff distance from $\pi^{-1}(Im \ \gamma)$. There exists a geodesic $\lambda \in  \Phi_{\varpi(p)}$ such that its image has finite Hausdorff distance from $Im \ \varpi \circ \gamma$. Then $Im \ \lambda$ and $\pi^{-1}(Im \ \gamma)$ have finite Hausdorff distance, which implies that $f^{\ast}(Im \ \lambda) = f^{\ast}(Cl_{Y}(\pi^{-1}(Im \ \gamma))) =  f(Cl_{X}(\pi(Cl_{Y}(\pi^{-1}(Im \ \gamma))))) = f(Cl_{X}(\pi(\pi^{-1}(Im \ \gamma)))) \subseteq f(Im \ \gamma) = \{x\}$. Then $f^{\ast}(Im \ \lambda) = \{x\}$, since $g(Im \ \gamma)$ is not empty (because $Im \ \gamma$ is not bounded and $Y+_{f^{\ast}}W$ is compact, by \textbf{Proposition \ref{compactglue}}).

Thus $\Pi(X+_{f}W)$ is $\Phi_{\varpi(p)}$-accessible.
\end{proof}

\begin{teo}\label{preserveshyperboliccompactification}Let $(X,d)$ and $(Y,d')$ be proper $\delta$-hyperbolic geodesic spaces and $\pi: Y \rightarrow X$ a quasi-isometry with quasi-inverse $\varpi$. Then the functor $\Pi$ sends the hyperbolic compactification of $X$ to the hyperbolic compactification of $Y$.
\end{teo}

\begin{proof}We have that the functor $\Pi$ sends compactifications $\Phi_{p}$-accessible to compactifications $\Phi_{\varpi(p)}$-accessible. By \textbf{Proposition \ref{changebasepoint}}, the functor $\Pi$ sends compactifications that are $\Phi_{p}$-accessible $\forall p \in X$ to compactifications that are $\Phi_{q}$-accessible to every $q \in Y$. Since $\Pi$ is an isomorphism of categories, it preserves the limits of these compactifications. So it sends the hyperbolic compactification of $X$ to the hyperbolic compactification of $Y$.
\end{proof}

As an immediate consequence we have another proof of the well known fact that if $X$ and $Y$ are quasi-isometric hyperbolic spaces, then they have homeomorphic hyperbolic boundaries.

\section{Relatively hyperbolic groups}

\begin{prop}\label{continuouscriteriaforisomorphism}Let $(X,\varepsilon)$ and $(Y,\eta)$be locally compact paracompact Hausdorff spaces with proper coarsely connected coarse structures and $\pi: Y \rightarrow X$ be a continuous coarse equivalence. If $X+_{f}Z$ and $Y+_{g}Z'$ are perspective compactifications such that the map $\pi$ extends to a continuous map $\pi+\psi: Y+_{g}Z' \rightarrow X+_{f}Z$, then there is a morphism in $Pers(\eta)$ given by $id+ \psi: Y+_{g}Z' \rightarrow \Pi(X+_{f}Z)$. Moreover, if $\psi$ is a homeomorphism, then $\Pi(X+_{f}Z)$ is isomorphic to $Y+_{g}Z'$ and if $Z = Z'$ and $\psi = id_{Z}$, then $\Pi(X+_{f}Z) = Y+_{g}Z$.
\end{prop}

\begin{proof}The map $\pi$ induces a continuous map $\pi+\psi: Y+_{g}Z' \rightarrow X+_{f}Z$. We also have that the map $\pi+id: Y+_{f^{\ast}}Z \rightarrow X+_{f}Z$ is continuous, where $f^{\ast}$ is the pullback of $f$ with respect to $\pi$ and $id_{Z}$. By the \textbf{Proposition \ref{eightlemma}}, the map $id+\psi: Y+_{g}Z' \rightarrow Y+_{f^{\ast}}Z$ is continuous. But $Y+_{f^{\ast}}Z = \Pi(X+_{f}Z)$. 

The rest is immediate.
\end{proof}

\begin{prop}\label{preserviingrelativelyhyperboliccompactification}Let $G$ and $H$ be finitely generated groups and $\pi: G \rightarrow H$ a quasi-isometry. Then the induced functor $\Pi: Pers(\varepsilon_{H}) \rightarrow Pers(\varepsilon_{G})$ sends the compactifications with the relatively hyperbolic compactifications of $H$ with respect to subgroups that are not relatively hyperbolic to relatively hyperbolic compactifications of $G$ with respect to subgroups that are not relatively hyperbolic to relatively hyperbolic .
\end{prop}

\begin{obs}In a fact, $\Pi$ gives a bijection between the relatively hyperbolic compactifications of $H$ with respect to subgroups that are not relatively hyperbolic to relatively hyperbolic  and the relatively hyperbolic compactifications of $G$ with respect to subgroups that are not relatively hyperbolic, since this proposition also holds for $\Pi^{-1}$. As it is clear in the following proof, this bijection is the one given by Theorem 6.3 of \cite{Gro}.
\end{obs}

\begin{proof}Let $(H, \mathcal{P})$ be a relatively hyperbolic pair. We know that the compactification $H+_{\partial_{c}}\partial_{B}(H,\mathcal{P}) \in EPers(H)$ (\textbf{Proposition \ref{convergenceisperspective}}), which implies that $H+_{\partial_{1c}}\partial_{B}(H,\mathcal{P}) \in Pers(\varepsilon_{H})$ (\textbf{Proposition \ref{coarsegroup})}. By Theorem 6.3 of \cite{Gro}, there exists $\mathcal{Q}$ such that $(G,\mathcal{Q})$ is a relatively hyperbolic pair and the map $\pi$ extends to a quasi-isometry $\pi': X(G,\mathcal{Q}) \rightarrow X(H,\mathcal{P})$, where $X(\Gamma, \mathcal{R})$ is the cusped space of the group $\Gamma$ with respect to a set of generators (that we omit here) and a set of subgroups $\mathcal{R}$ (see \cite{GrM} for the definition). Since $(G,\mathcal{Q})$ and $(H, \mathcal{P})$ are relatively hyperbolic pairs, we have that $X(G,\mathcal{Q})$ and $X(H,\mathcal{P})$ are hyperbolic (Theorem 3.25 of \cite{GrM}) and their hyperbolic boundaries are the Bowditch boundaries of the pairs. So the map $\pi'$ induces a continuous map $\pi+\psi: G+_{\partial_{2c}}\partial_{B}(G,\mathcal{Q}) \rightarrow H+_{\partial_{1c}}\partial_{B}(H,\mathcal{P})$, where $\psi$ is a homeomorphism (given by Corollary 6.5 of \cite{Gro}). Then, by \textbf{Proposition \ref{continuouscriteriaforisomorphism}}, the map $id+\psi: G+_{\partial_{2c}}\partial_{B}(G,\mathcal{Q}) \rightarrow \Pi(H+_{\partial_{1c}}\partial_{B}(H,\mathcal{P}))$ is a homeomorphism.
\end{proof}

\section{Floyd compactifications}

\begin{prop}(Changing maps) Let $(X,\varepsilon)$ and $(Y,\eta)$ be locally compact paracompact Hausdorff spaces with proper coarsely connected coarse structures, $X+_{f}Z$ and $Y+_{g}Z'$ compactifications with the perspectivity property and $\pi+\varpi: Y+_{g}Z' \rightarrow X+_{f}Z$ a continuous map. If $\pi': Y \rightarrow X$ is a continuous map that is close to $\pi$, then the map  $\pi'+\varpi: Y+_{g}Z' \rightarrow X+_{f}Z$ is continuous.
\end{prop}

\begin{proof}By \textbf{Proposition \ref{closemapsequalpullbacks}} the pullback $f^{\ast}$ of $f$ with respect to $\pi$ and $id_{Z}$ is equal to the pullback of $f$ with respect to $\pi'$ and $id_{Z}$. By the definition of pullback and by \textbf{Proposition \ref{eightlemma}}, we have that all the maps of the following diagram are continuous:

$$ \xymatrix{  Y+_{g}Z' \ar[r]_{\pi+\varpi} \ar[d]_{id+id} & X+_{f}Z  \\
                Y+_{f^{\ast\ast}}Z' \ar[r]^{id+\varpi} & Y+_{f^{\ast}}Z \ar@<4pt>[u]^{\pi+id} \ar@<-4pt>[u]_{\pi'+id} } $$

Thus, the map  $\pi'+\varpi: Y+_{g}Z' \rightarrow X+_{f}Z$ is continuous.
\end{proof}

\begin{prop}\label{extensionfloyd}(Gerasimov-Potyagailo, Lemma 5.4 of \cite{GP}) Let $\Gamma_{1}$ and $\Gamma_{2}$ be locally finite connected graphs, $(X_{i},d_{i})$ the set of vertices of $\Gamma_{i}$ with its induced metric and $\pi: X_{2} \rightarrow X_{1}$ a $\alpha$-quasi-isometric map, where $\alpha: \N \rightarrow\N$ is a distorted map satisfying:

\begin{enumerate}
    \item $\exists D > 0$ such that $\forall n \in \N$, $\frac{f_{2}(n)}{f_{1}(\alpha(n))} \leqslant D$.
\end{enumerate}

Then we have that the map $\pi$ extends to a continuous map $\\ \pi+\psi: X_{2}+_{\partial_{f_{2}}}\partial_{f_{2}}(X_{2}) \rightarrow X_{1}+_{\partial_{f_{1}}}\partial_{f_{1}}(X_{1})$.
\end{prop}

\begin{prop}\label{extensionfloydhomeo} Let $\Gamma_{1}$ and $\Gamma_{2}$ be locally finite connected graphs, $(X_{i},d_{i})$ the set of vertices of $\Gamma_{i}$ with its induced metric and $\pi: X_{2} \rightarrow X_{1}$ a $\alpha$-quasi-isometry, where $\alpha: \N \rightarrow\N$ is a distorted map satisfying:

\begin{enumerate}
    \item $\exists D > 0$ such that $\forall n \in \N$, $\frac{f_{2}(n)}{f_{1}(\alpha(n))} \leqslant D$ and $\frac{f_{1}(n)}{f_{2}(\alpha(n))} \leqslant D$.
\end{enumerate}

Let  $\pi+\psi: X_{2}+_{\partial_{f_{2}}}\partial_{f_{2}}(X_{2}) \rightarrow X_{1}+_{\partial_{f_{1}}}\partial_{f_{1}}(X_{1})$ be the continuous extension of $\pi$. Then $\psi$ is a homeomorphism.
\end{prop}

\begin{obs}Here we need that $\pi$ is a $\alpha$-quasi-isometry, i.e., it is a $\alpha$-quasi-isometric map, it has a quasi-inverse and this quasi-inverse is also a $\alpha$-quasi-isometric map.
\end{obs}

\begin{proof}Let $\varpi$ be the quasi-inverse of $\pi$. Then the map $\varpi$ also extends to a continuous map $\varpi+\phi:  X_{1}+_{\partial_{f_{1}}}\partial_{f_{1}}(X_{1}) \rightarrow X_{2}+_{\partial_{f_{2}}}\partial_{f_{2}}(X_{2})$, which implies that the map $(\varpi+\phi) \circ (\pi+\psi): X_{2}+_{\partial_{f_{2}}}\partial_{f_{2}}(X_{2}) \rightarrow X_{2}+_{\partial_{f_{2}}}\partial_{f_{2}}(X_{2})$ is continuous. Since the map $\varpi \circ \pi$ is close to $id_{X_{2}}$, we have that the map $id+(\phi \circ \psi): X_{2}+_{\partial_{f_{2}}}\partial_{f_{2}}(X_{2}) \rightarrow X_{2}+_{\partial_{f_{2}}}\partial_{f_{2}}(X_{2})$ is continuous. But $X_{2}$ is dense on $X_{2}+_{\partial_{f_{2}}}\partial_{f_{2}}(X_{2})$, which implies that $\phi \circ \psi = id_{\partial_{f_{2}}(X_{2})}$. Analogously, $\psi\circ \phi = id_{\partial_{f_{1}}(X_{1})}$, which implies that $\psi$ is a homeomorphism.
\end{proof}

\begin{cor}\label{preservingfloydcompactifications}Let $\Gamma_{1}$ and $\Gamma_{2}$ be locally finite connected graphs, $(X_{i},d_{i})$ the set of vertices of $\Gamma_{i}$ with its induced metric  and $\pi: X_{2} \rightarrow X_{1}$ a $\alpha$-quasi-isometry, where $\alpha: \N \rightarrow\N$ is a distorted map satisfying:

\begin{enumerate}
    \item $\exists D > 0$ such that $\forall n \in \N$, $\frac{f_{2}(n)}{f_{1}(\alpha(n))} \leqslant D$ and $\frac{f_{1}(n)}{f_{2}(\alpha(n))} \leqslant D$.
\end{enumerate}

Then the induced functor $\Pi$ sends the Floyd compactification of $X_{1}$ to a space that is isomorphic to the Floyd compactification of $X_{2}$.
\end{cor}

\begin{proof}By \textbf{Proposition \ref{extensionfloydhomeo}}, we have that the map $\pi$ induces a continuous map $\pi+\psi: X_{2}+_{\partial_{f_{2}}}\partial_{f_{2}}(X_{2}) \rightarrow X_{1}+_{\partial_{f_{1}}}\partial_{f_{1}}(X_{1})$, where $\psi$ is a homeomorphism. So, by \textbf{Proposition \ref{continuouscriteriaforisomorphism}}, $X_{2}+_{\partial_{f_{2}}}\partial_{f_{2}}(X_{2})$ is equivalent to $\Pi(X_{1}+_{\partial_{f_{1}}}\partial_{f_{1}}(X_{1}))$.
\end{proof}

By Proposition 2.4 of \cite{GP}, the Floyd compactification of a locally finite graph is $\Phi_{v}$-accessible for any vertex $v$. As a consequence, we have the following:

\begin{prop}Let $\Gamma$ be a locally finite connected graph, $(X,d)$ the set of vertices of $\Gamma$ with its induced metric and $f$ a Floyd map. Then there exists a surjective map $id+\varpi: X+_{f_{\infty}}\partial_{\infty}(X) \rightarrow X+_{\partial_{f}}\partial_{f}(X)$. \eod
\end{prop}

\begin{obs}The version of this proposition for groups is well known. It is a consequence of Proposition 3.4.6 of \cite{Ge2}.
\end{obs}

\end{document}